\documentclass[%
 aip,
 amsmath,amssymb,
 reprint,%
]{revtex4-1}
\usepackage[utf8]{inputenc}
\usepackage{enumitem}
\usepackage{varwidth}
\usepackage{tasks}
\usepackage{float}
\usepackage{wrapfig,lipsum,booktabs}
\usepackage{amsfonts,amsmath,amssymb,mathtools}
\usepackage{graphicx}
\usepackage{float,color}
\usepackage{caption}
\usepackage[english]{babel}
\usepackage{ulem}
\usepackage{tabularx}
\usepackage{hyperref}
\usepackage{amsthm}
\usepackage{ stmaryrd }
\usepackage{amssymb}
\newtheorem{theorem}{Theorem}
\newtheorem{lemma}{Lemma}
\newtheorem{cor}{Corollary}
\newtheorem{prop}{Proposition}
\newtheorem{defn}{Definition}
\usepackage{graphicx}
\usepackage{dcolumn}
\usepackage{bm}

\usepackage[utf8]{inputenc}
\usepackage[T1]{fontenc}
\usepackage{mathptmx}
\usepackage{etoolbox}

\makeatletter
\def\@email#1#2{%
 \endgroup
 \patchcmd{\titleblock@produce}
  {\frontmatter@RRAPformat}
  {\frontmatter@RRAPformat{\produce@RRAP{*#1\href{mailto:#2}{#2}}}\frontmatter@RRAPformat}
  {}{}
}%
\makeatother
\begin{document}

\preprint{AIP/123-QED}

\title[On the Spectrum of the Periodic Anderson-Bernoulli Model]{On the Spectrum of the Periodic Anderson-Bernoulli Model}
\author{William Wood}

 \email{woodwf@uci.edu}
\affiliation{ 
Department of Mathematics at University of California, Irvine
}%
\newcommand{\edited}[1]{\textcolor{black}{#1}}

\date{\today}

\begin{abstract}
We analyze the spectrum of a discrete Schr\"{o}dinger operator with a potential given by a periodic variant of the Anderson Model.  In order to do so, we study the uniform hyperbolicity of a Schr\"odinger  cocycle generated by the $SL(2,\mathbb{R})$ transfer matrices.  In the specific case of the potential generated by an alternating sequence of random values we show that the almost sure spectrum consists of at most 4 intervals. 
\end{abstract}

\maketitle

\section{\label{sec:1}Introduction}

Discrete Schr\"{o}dinger operators $H_{\edited{V}}:l^2(\mathbb{Z})\rightarrow l^2(\mathbb{Z})$ 
\begin{equation}\label{eq:1}
    H_{\edited{V}} \phi(n)=\phi(n+1)+\phi(n-1)+V(n)\phi(n)
\end{equation} are the focus of this paper.  Here,  $V=\{V(n)\}_{n\in\mathbb{Z}}$ is a bounded, real sequence, and the spectrum of $H_V$ is the set \[\edited{\sigma(H_V)=\{\lambda: H_V-\lambda\text{ does not have a bounded inverse}\}.}\]  \edited{The spectral type (absolutely continuous, singular continuous, pure point, or mixed) and the topological structure of the spectrum of these operators are objects of study}, see \cite{d1} for more detail.  \edited{Since} \hyperlink{eq:1}{(1)} is a bounded, self-adjoint operator in a Hilbert space, its spectrum $\sigma(H_V)$ is a compact subset of $\mathbb{R}$.  If $V$ is a periodic sequence with period $m$, the spectrum consists of, at most, $m$ intervals.  \edited{If the potential is constant,} the spectrum is the set $[-2,2]+V(0)$. 

There are plenty of more complicated examples, such as the Almost Mathieu  Operator.  If $V_{AM}(n)=\lambda\cos(n\pi\alpha)$ where \edited{$\lambda\neq0,$} $\alpha\notin\mathbb{Q},$ and $n\in\mathbb{Z}^{\edited{+}}$, then $\sigma(H_{AM})$ is a Cantor set.  The spectrum was proven to be a Cantor set by Avila and Jitomirskaya\edited{\cite{AJ}}, and this problem was known as the Ten Martini Problem.  A nice survey of the subject was provided by Marx and Jitormiskaya \cite{mj}.  Similarly, if the potential is given by the Fibonacci substitution sequence, the spectrum is a Cantor set as well\edited{\cite{d1,d4}}.

The last example is the Anderson Model, \edited{where the potential depends on random variables}.  If $V(n)$ is a random potential, given by a sequence of iid random variables, then there exists an almost-sure spectrum $\sigma_{as}$, and we have \[\sigma_{as}(H_V)=\sigma(\triangle)+supp\;\mu\]
Here $\triangle$ is the free Laplacian,  $\sigma(\triangle)=[-2,2]$, \edited{ and $\mu$ is the random variable}.  Additional details \edited{about} this model can be found in \cite{s}. 

This paper will address a model inspired by the Anderson Model.  If the sequence \edited{$\{V(n)\}_{n\in\mathbb{Z}}$} is given by random variables that are not identically distributed, how can the spectrum be characterized?
\\
{\bf Question:}
\edited{\it Suppose that $m\in\mathbb{N},$ and $\nu_0, \nu_2,\cdots\;\nu_{m-1} $ are probability distributions  on $\mathbb{R}$ with finite support.  Suppose that a random potential $\{V(n)\}$ is given by a sequence of independent random variables, such that $V(n)$ is distributed with respect to $\nu_n\;(mod\;m)$.  Does the spectrum $\sigma(H_V)$ almost surely consist of a finite number of intervals?}

\edited{To the best} of my knowledge, this problem is open.  This paper addresses one specific case, when $m=2$.

Let us denote by $B(p)$ a Bernoulli random variable that takes value 1 with probability \edited{$p\in(0,1)$}, and value 0 with probability $1-p$.

\begin{theorem}\label{thm:1}
For any real $\lambda_0,\lambda_1,c_0,c_1,$ a discrete Schr\"{o}dinger operator with random potential, $V(n),$  defined by the sequence of distributions:
\begin{equation}\label{eq:2}\nu(n)=\begin{cases} \lambda_0 B(p_0)+c_0 & n\text{ is even} \\ \lambda_1 B(p_1)+c_1 & n\text{ is odd}\end{cases}\end{equation}\\
the almost sure spectrum $\sigma_{as}(H_{\edited{V}})$ consists of at most 4 closed intervals
\end{theorem}
\edited{Notice that if $\lambda_0=0$ or $\lambda_1=0$, then the potential has a constant element.  In section \hyperref[sec:4E]{$4E$}} we provide an explicit formula for the spectrum in terms of the parameters $\lambda_0,\lambda_1,c_0,\;\&\;c_1$.

After this paper was written, I became aware of the preprint \cite{d5} "Spectral Characteristics of Schr\"odinger operators generated by product systems" written by David Damanik, Jake Fillman, and Philipp Gohlke who independently proved the same  result. We agreed to post our texts to ArXiv at the same time. 
\section{\label{sec:2}Preliminaries}
Here, we are going to examine the uniform hyperbolicity of the cocycle generated by transfer matrices.  Specifically, we will look at the transfer matrices of the Schr\"{o}dinger operator with potential given by \hyperref[eq:2]{$(2)$}.  The concept of uniform hyperbolicity is presented in subsection \hyperref[sec:2A]{2A}.  Uniform hyperbolicity has similar, yet distinct, definitions when applied to cocyles and sets of matrices.  In the context of this paper, the specific cocycle $(T,A_E)$ generated by the transfer matrices is found to be uniformly hyperbolic if and only if a related set of matrices is uniformly hyperbolic.

In subsection \hyperref[sec:2B]{2B}, Johnson's Theorem and transfer matrices are introduced.  \hyperref[thm:3]{Johnson's Theorem}\edited{\cite{z,jo}} states that an energy $E$ belongs \edited{ to the almost sure spectrum of the ergodic family of operators $\{H_V\}$, $E\in\sigma_{as}(H_V)$}, if and only if the relevant cocycle $(T,A_E)$ (\edited{generated by} the transfer matrices) is \edited{not} uniformly hyperbolic.  For a brief history of Johnson's Theorem, and how it originated in a paper by Russell Johnson in 1986, see \cite{z2} and \edited{\cite{jo}}.  Depending on how well the potential $\edited{V}(n)$ is understood, the respective cocycle can be analyzed. In section \hyperref[sec:2C]{\edited{$2C$}}, speculations about the broader case mentioned in the above question will be made \edited{and shown to meet Johnson's Theorem's conditions.  Lemma \hyperref[lem:1]{1} is also introduced and proven, which states the uniform hyperbolicity of the respective cocycle is equivalent to the uniform hyperoblicity of a related set of matrices.}   \edited{Last, in section \hyperref[sec:2D]{$2D$}, we hypothesize about the question from section \hyperref[sec:1]{$1$}.}

In section \hyperref[sec:3]{3}, lemma \hyperref[lem:2]{2}, which is key to proving Theorem \hyperref[thm:1]{1}, is outlined, \edited{ The poof of the lemma, however, is provided in section \hyperref[sec:4D]{$4D$}.  Lemma \hyperref[lem:2]{$2$} states an equivalence exists between the  relevant set of matrices being uniformly hyperbolic and a condition that is much easier to calculate.}  

\edited{ In section \hyperref[sec:4]{4}, we go into proving lemma 2.  The first three subsections are dedicated to providing necessary lemmas.  The second to last subsection goes into proving lemma \hyperref[lem:2]{$2$}, and the last subsection explicitly defines the interval the complement of the spectrum. }

 \subsection{\label{sec:2A}Uniform Hyperbolicity}

First, we will need to define what it means that a cocycle or a set of matrices is uniformly hyperbolic\edited{ \cite{aby,y,ca}}.  \edited{We will begin by talking about sets of matrices.}
\begin{defn}
\label{def:1}A set of matrices $\{A_n\}\subset SL(2,\mathbb{R})$ is \emph{uniformly hyperbolic} if there exists the constant $\lambda>1$ such that, for any product $\prod A_i$ of $m$ matrices, the inequality $||\prod A_i||>\lambda^m$ holds.
\end{defn}  \edited{In this product, duplicates are allowed.  Due to Avila, Bochi, and Yoccoz} \cite{aby}, we have the following Theorem. \begin{theorem}\label{thm:2}
A \edited{finite} set of $2\times 2$ matrices $\{A_i\}$ is uniformly hyperbolic if and only if there is a nontrivial open set \edited{$C\subsetneq\mathbb{P}^1$ (hereafter referred to as a cone) such that $\overline{C}\subsetneq\mathbb{P}^1$} is a finite number of intervals and $A_i(\overline{C})\subset C$ for each matrix $A_i$ from the set.  
\end{theorem}
For this paper, the projective space, $\mathbb{P},$ will be parameterized by $\begin{bmatrix}x \\ y\end{bmatrix} \leftrightarrow \frac{x}{y}$.  The simplest example of a uniformly hyperbolic set of matrices is when $C$ is a principal cone. \begin{defn} If the cone $C$ consists of a single interval, it is called a \emph{principal cone}.  \end{defn} If the set $\{A_i\}$ is a set of $2\times 2$ matrices, a principal cone forms if and only if all the unstable eigenvectors (here the unstable eigenvector of $A_i$ will be expressed as \edited{ $u_i$}) lie in a single interval in the projective space that contains none of the stable eigenvectors \edited{(here the stable eigenvector of $A_i$ will be expressed as $s_i$).  The cone, $C$, can be defined as that interval.}  

Before covering the second definition of uniform hyperbolicity, we need to define a cocyle. \begin{defn} Given a \edited{metric} space $\Omega$, \edited{homeomorphism} $T:\Omega\rightarrow \Omega$ and \edited{continuous} $B:\Omega\rightarrow SL(2,\mathbb{R})$, a \emph{cocycle} is the  map\[(T,B):(\Omega,\mathbb{R}^2)\rightarrow(\Omega,\mathbb{R}^2)\]  \[(\omega,\vec{v})\mapsto(T(\omega),B(\omega)\vec{v}).\] \end{defn}  \begin{defn}\label{def:4}

Denoting $(T,B)^n=(T^n,B_n)$ where \[B_n(\omega)=B(T^{n-1}(\omega))\cdot B(T^{n-2}(\omega))\cdots B(T(\omega))B(\omega).\] A cocycle $(T,B)$ is \edited{\emph{uniformly hyperbolic}} if there exist $\lambda>1$ and $C>0$ such that $||B_n(\omega)||> C \lambda^n$ for all $n\in\mathbb{Z}^+$ and $\omega\in\Omega$. 

\end{defn} \edited{More general definitions of hyperbolicity are available\cite{y}, but we do not use them here.}

\subsection{\label{sec:2B}Johnson's Theorem and Transfer Matrices}

Given a dynamical system $(\Omega,T,\mu)$  \edited{ with compact metric} space $\Omega$, probability measure $\mu$, ergodic $T$, and a \edited{continuous} function $f:\Omega\rightarrow\mathbb{R}$, an ergodic potential \edited{$V$} of a Schr\"{o}dinger operator is defined by the equation \begin{equation}\label{eq:3} \edited{V}_\omega(n)=f(T^n\omega)\end{equation} for some $\omega\in\Omega$.  To provide a simple example, if $\Omega$ is a finite set and $T$ is a bijection, then $\edited{V}_\omega(n)$ would be a  periodic potential.  Any potential with period m can be defined as such with $\Omega=\mathbb{Z}_m$. 

In order to study \edited{operators with} an ergodic potential, it would be useful to define a transfer matrix.  First, \edited{if, the continuous functions} $f:\Omega\rightarrow\mathbb{R}$ and $g:\mathbb{R}\rightarrow SL(2,\mathbb{R})$ are given, then $(T,g\circ f(\cdot))$ would be a $SL(2,\mathbb{R})-$cocycle.  Here $\edited{(T,g\circ f)}:(\Omega,\edited{\mathbb{R}^2})\rightarrow(\Omega,\edited{\mathbb{R}^2})$ \edited{acts by} $(\omega,\vec{v})\mapsto(T\omega,g\circ f(\omega)\cdot \vec{v})$.  Next, given some \edited{Schr\"odinger} operator $H$ with ergodic potential $\edited{V_\omega}$, assuming there is some $E$ such that $E\phi(n)=H_\omega\phi(n)$ \edited{ and} taking into account equation \hyperlink{eq:1}{$(1)$}, we have
\[\begin{bmatrix} \phi(n+1) \\ \phi(n)  \end{bmatrix}=\begin{bmatrix} E-\edited{V_\omega}(n) & -1 \\ 1 & 0 \end{bmatrix}\cdot \begin{bmatrix}\phi(n) \\ \phi(n-1)\end{bmatrix}\]   We arrive at the transfer matrix 
\begin{equation}\label{eq:4}A_E(n)=g_E\circ \edited{V_\omega}(n)=\begin{bmatrix} E-\edited{V_\omega}(n) & -1 \\ 1 & 0 \end{bmatrix}\end{equation} \edited{where $g_E:\mathbb{R}\rightarrow SL(2,\mathbb{R})$ is given by \[g_E(x)=\begin{bmatrix}E-x & -1\\1 &0\end{bmatrix},\]and $\edited{V_\omega}(n)$ is the potential at the $n^{th}$ place.}
Transfer matrices \edited{are $SL(2,\mathbb{R})$ matrices} and satisfy the property
\[\begin{bmatrix} \phi(m+1) \\ \phi(m)\end{bmatrix} = A_E(m)\cdot A_E(m-1)\cdots A_E(n)\begin{bmatrix} \phi(n)\\ \phi(n-1)\end{bmatrix}\]\edited{assuming $E\phi=H_\omega\phi$.}  Define \edited{\[ \footnotesize A^m(n,E)=\begin{cases} A_E(m+n-1)\cdot A_E(m+n-2)\cdots A_E(n) & \text{ if } m>0  \\ I & \text{ if }m=0 \\ A_E(n-m)^{-1}\cdots A_E(n-2)^{-1}\cdot A_E(n-1)^{-1} & \text{ if } m<0 \end{cases} \]  Denote $A^m_E=A_E(m-1)\cdot A_E(m-2)\cdots A_E(1)\cdot A_E(0)$.  }

In order to study the spectrum of the Schr\"{o}dinger operator, we will use the notion of uniform hyperbolicity in definition 4 applied to  Schr\"{o}dinger cocycles.

We have Johnson's Theorem\cite{z,jo}:

\begin{theorem}[Johnson's Theorem]\label{thm:3}

If $(\Omega,T,\mu)$ is an ergodic dynamical system, $\Omega$ is  a compact metric space, $T$ is a \edited{homeomorphism}, $\mu$ is a $T$-invariant measure with the support of the measure being $\Omega$ and $f:\Omega\rightarrow\mathbb{R}$ is continuous, then the almost sure spectrum $\sigma_{as}(H)$ of the operator \hyperref[eq:1]{(1)} with an ergodic potential defined by this dynamical system is given by  \[\sigma_{as}(H)=\{E:(T,A_E) \text{ is not uniformly hyperbolic}\}\]
\end{theorem}

\subsection{\label{sec:2C}Hyperbolicity Locus and Anderson-Bernoulli Model}
  The paper's main result, Theorem \hyperref[thm:1]{1}, applies to a specific random potential \edited{ \hyperref[eq:2]{$(2)$}}.  This opens the question \edited{about a more general class of potentials.  Specifically, how would we characterize the spectrum if the potential} is defined by a  periodic sequence of more than two distributions \edited{ as defined in the question in section \hyperref[sec:1]{$1$}}.  The model would become more complicated and some of the tools in this paper will not apply.  \edited{In fact, I would conjecture} that with large enough period, a result similar to Theorem \hyperref[thm:1]{$1$} would not hold. 
  
  Denote a random potential as being defined by \edited{$m$ distributions like:} \begin{equation}\label{eq:5}\nu(n)=\begin{cases}\lambda_0B(p_0)+c_0 & n\equiv 0 \mod m \\ \lambda_1B(p_1)+c_1 & n\equiv 1\mod m \\ \hspace{2em}\vdots & \hspace{1em} \vdots \\ \lambda_{m-1}B(p_{m-1})+c_{m-1} & n\equiv m-1 \mod m \end{cases}\end{equation}
  Here, $\lambda_0,\lambda_1,\cdots,\lambda_{m-1}$ and $c_0,c_1,\cdots,c_{m-1}$ are real numbers, $m$ is fixed, and the Bernoulli distribution $B$ takes on the values $0$ or $1$ with probability $0<p_i<1$ $(i=0,1,\cdots,m-1)$.  \edited{Denote $X=\{0,1\}^\mathbb{Z},$  $\Omega=\mathbb{Z}_m\times X,$} and define the cocycle \[\edited{(T,A_E):\Omega\times\mathbb{R}^2\rightarrow \Omega\times\mathbb{R}^2.}\] \edited{  Assuming $\omega=(k,x)\in\Omega,$ then we can define $k\in \mathbb{Z}_m$ and $x\in X$, and furthermore, we can be define $x$ as a sequence $(\cdots x_{n-1},x_n,x_{n+1},\cdots)$ such that $x_i$ is $0$ or $1$.  Defining $T_1:X\rightarrow X$ to be the shift operator such that $(T_1x)_n=x_{n+1}$, we can define} \begin{equation}\label{eq:6}\begin{aligned} T:(k,x)\mapsto(k+1 \;mod\; m,\;T_1(x))\\ (T,A_E)(\omega,\vec{v})=(T(\omega),A_{E,\omega}\vec{v}).\end{aligned}\end{equation}   \edited{Let $A_{E,\omega}$ be a transfer matrix from \hyperref[eq:4]{\edited{$(4)$}} corresponding to potential $V_{\omega}(k)$ from \hyperref[eq:3]{\edited{$(3)$}}.  We \edited{will denote} the potential defined by $\nu$ from \hyperref[eq:5]{\edited{$(5)$}} as $V_{(0,x)}$ and show it is dynamically defined and corresponds to an ergodic family of operators.  We have:} \[\begin{array}{c}  \edited{f((k,x))=\lambda_k\cdot x_0+c_k}\\ \\  V_{(0,x)}(0)=f(0,x)=\edited{\lambda_0\cdot x_0+c_0}\\ \\ V_{(0,x)}(n)=f(T^n(0,x))=(\lambda_{n\;mod\; m} )\cdot x_n+(c_{n\;mod\; m})\\ \\ A_{E,(k,x)}=\begin{bmatrix} E-\edited{V_{(0,x)}(k)} & -1 \\ 1 & 0\end{bmatrix}\\ \\ \edited{(T,A_E):\big((k,x) , \vec{v} \big)}\mapsto \big( (k+1\;mod\;m,T_1x)  ,  A_{E,(k,x)}\cdot \vec{v}  \big) \end{array}\]\edited{Notice that we only need to consider $V_\omega$ with $\omega=(0,x)$, as $V_{k,x}$ would involve shifting the distributions of $\nu$ by $k$ positions.}

  \edited{ We can use the counting, probability measure $\mu_0$ on $\mathbb{Z}_m$.  Using probabilities $\{p_i\}_{i=1}^m$ from  \hyperref[eq:5]{$(5)$}, probability measure $\mu_1$  over $X$ can be defined by \[\mu_1(\{x:x_n=1\})=p_{n\mod m},\] and finally $\mu$ over $\Omega$ can be the product of probability measures $\mu_0\times\mu_1$.  The systems $(\Omega_0,T_0)$ and $(\Omega_1,T_1)$ are ergodic, and due to $T_1$ being mixing, $(\Omega, T)$ is ergodic.  Additionally, $\Omega$ is compact and $T-$invariant.  The cocycle $(T,A_E),$ is thus a Schr\"{o}dinger cocycle of an ergodic family of operators.  Because $(\Omega,T)$ is ergodic and satisfies the conditions of Theorem 3 (\hyperref[thm:3]{Johnson's Theorem}), the almost sure spectrum is equal to the set} \[\edited{\sigma_{as}(H_V)=\{E:(T,A_E)\text{ is not uniformly hyperbolic}\}}.\]   \edited{The following lemma can be used to better analysis when the respective cocycle is uniformly hyperbolic.}

  \begin{lemma}\label{lem:1} 
Consider an ergodic family of Schr\"{o}dinger operators with random potentials expressible as a possible outcome of the distribution \hyperref[eq:5]{\edited{$(5)$}}. The $SL(2,\mathbb{R})-$cocycle $(T,A_E)$ is uniformly hyperbolic if and only if the set of distinct matrices \[\{M:M=A^m_{E,(0,x)}\text{ for some }(0,x)\in\edited{\Omega}\}\] is uniformly hyperbolic \edited{for any $m\in\mathbb{Z}^+$}.
\end{lemma}
 Remark: There are \edited{at most} 2 distinct elements in $\{A_{E,(0,x)}:\text{ for some }(0,x)\in\edited{\Omega}\}$.  These elements are expressible as
 \[\begin{bmatrix}E-\lambda_0x_0-c_0 & -1 \\ 1 & 0\end{bmatrix}\] for $x_0\in\{0,1\}$.  Similarly, $\{\edited{A^2_{E,(0,x)}:(0,x)\in\Omega}\}$ has at most 4 distinct elements, expressible as \[\begin{bmatrix}E-\lambda_0x_0-c_0 & -1 \\ 1 & 0\end{bmatrix}\cdot\begin{bmatrix} E-\lambda_1x_1-c_1 & -1 \\ 1 & 0\end{bmatrix}\] for $x_0,x_1\in\{0,1\}$. The set $\{A^m_{E,(0,x)}:(0,x)\in\edited{\Omega}\}$ has at most $2^m$ matrices.  

\begin{proof}\

  $(\leftarrow)$ If the set of matrices is uniformly hyperbolic, then there exists $\lambda>1$  such that for any ordered set $I$ of the elements in $\{M:M=A^m_{E,(0,x)}\text{ for some }(0,x)\in\Omega\}$ where $|I|=k$,  $\big|\big|\prod_{ I}  M_i\big|\big|>\lambda^k$.  \edited{Here $I$ can have repeated elements.}  For any $(0,x)\in\Omega$ and any $k$, there exists $|I|=k$, such that $\prod_{I} M_i=A^{km}_{E,(0,x)}.$ Therefore \[\Big|\Big|A^{km}_{E,(0,x)}\Big|\Big|=\Big|\Big| \edited{\prod_{I} M_i } \Big|\Big|>\lambda^{k}=\sqrt[m]{\lambda}^{mk}.\]
\edited{It is necessary, however, to bound $\Big|\Big|A^j_{E,(0,x)} \Big|\Big|$ for any $j$.  Next, a bound for $\Big|\Big|A^{j}_{E,(0,x)}\Big|\Big|$ will be provided for $0< j<m.$   Given that for fixed $E,(0,x),$ and $j\in(0,m)$, \edited{we have} \[\inf_{||\vec{v}||=1}\Big|\Big|A^j_{E,(0,x)}(\vec{v})\Big|\Big|=\Big|\Big|\big(A^j_{E,(0,x)}\big)^{-1}\Big|\Big|^{-1}\in(0,1)\]  Keeping $E$ fixed, there are at most $\sum_{0< j<m}2^j$ matrices expressible as $A^j_{E,(0,x)}$.  So, there exists $C$ such that
 \[\min_{\substack{ 0< j< m\\ (0,x) \in\Omega}} \inf_{||\vec{v}||=1} \Big|\Big|A^{j}_{E,(0,x)}(\vec{v})\Big|\Big|\Big/\sqrt[m]{\lambda}^j=C>0.   \]\edited{F}or $j\in(0,m),$ \edited{this} gives the inequality \[\inf_{||\vec{v}||=1}\Big|\Big|A^j_{E,(0,x)}(\vec{v})\Big|\Big|\geq C\sqrt[m]{\lambda}^j,\] and this brings the conclusion that for all $x\in X$:\[
\begin{array}{c} \Big|\Big|A^{mk+j}_{E,(0,x)}\Big|\Big|=\sup_{||\vec{v}||=1}\Big|\Big|A^{j}_{E,(mk\;mod\;m,x)} \cdot A^{mk}_{E,(0,x)} \vec{v} \Big|\Big|\geq\\ \\   \inf_{||\vec{v}||=1} \Big|\Big|A^{j}_{E,(0,x)}(\vec{v})\Big|\Big|\cdot \Big|\Big|A^{mk}_{E,(0,x)}\Big|\Big| > C\sqrt[m]{\lambda}^{mk+j}. \end{array}\]
 By definition 4, the cocycle is uniformly hyperbolic. }

$(\rightarrow)$ 
Assume the cocycle $(T,A)$ is uniformly hyperbolic,  then there exists a $\lambda>1$ and a $C>0$ such that for any $\omega\in\Omega$
\[\Big|\Big|A^k_{E,\omega} \Big|\Big|> C\lambda^{k}\]  

For any ordered set $I$ (with $|I|=k$) consisting of the matrices from the set $\{M_i\}$, there exists $\edited{(0,x)}\in\Omega$ such that $A_{E,\edited{(0,x)}}^{mk}=\prod_I M_i$.  Therefore,\[ \begin{array}{c} C\lambda^{mk}<\big|\big|\prod_I M_i\big|\big| \Rightarrow C\big(\lambda^m\big)^k<\big|\big|\prod_I M_i\big|\big|\end{array}.\] \edited{ For a set of matrices to be uniformly hyperbolic, we technically need to express such an inequality without a constant $C$.  To address this, we note that for any $\prod_i M_i,$ there exists $(0,x)\in\Omega$ such that for any $n\in\mathbb{Z}^+$, $(A_{E,(0,x)}^{nmk})=(\prod_I M_i)^n$, and we have}
\edited{\[ C\lambda^{nmk}<\Big|\Big|A^{nmk}_{E,(0,x)}\Big|\Big|=\Big|\Big|\Big(\prod_iM_i\Big)^n\Big|\Big|\leq \Big|\Big|\prod_iM_i\Big|\Big|^n.\]Taking the $n^{th}$ root,\[ C^{1/n}(\lambda^m)^k<\Big|\Big|\prod_iM_i\Big|\Big|,\] and allowing $n$ to be arbitrarily, we get large,\[(\lambda^m)^k\leq \Big|\Big|\prod_iM_i\Big|\Big|\Rightarrow (\lambda^{m/2})^k< \Big|\Big|\prod_iM_i\Big|\Big|.\]Therefore, the cocycle being uniformly hyperbolic \edited{from definition \hyperref[def:4]{$4$}} implies the set of matrices is uniformly hyperbolic \edited{from definition \hyperref[def:1]{$1$}}.}

 \end{proof}

\subsection{\label{sec:2D}Conjecture}


\edited{By \hyperref[thm:3]{Johnson's Theorem} and lemma \hyperref[lem:1]{$3$}, the almost-sure spectrum $\sigma(H_V)$ consisting of an infinite number of intervals is equivalent to the set of matrices being uniformly hyperbolic over an infinite number of intervals over $E.$  Theorem \hyperref[thm:1]{$1$} states that if the potential of an ergodic family is given by \hyperref[eq:5]{$(5)$} and $m=2$, the almost sure spectrum is a finite number of intervals. For larger $m$ this theorem does not apply. Using the paper by Avila, Bochi, and Yoccoz \cite{aby}, we can consider a geometric approach to this question. There exists open region $\mathcal{H}\subset SL(2,\mathbb{R})$ defined by} \[\edited{\scriptsize \mathcal{H}=\{(A_1,\cdots,A_m)\in SL(2,\mathbb{R})^m:\edited{\{A_i\}_{i=1}^{m}}\text{ is uniformly hyperbolic} \}.}\] \edited{We can parameterize a path in $SL(2,\mathbb{R})^{2^m}$ by $E$ and identify the intervals where the path intersects $\mathcal{H}$.  If the set of matrices defined in lemma \hyperref[lem:1]{1} parametized by $E$ define a path in $SL(2,\mathbb{R})^{2^m}$ and the spectrum consists of infinitely many intervals, then a path must intersect the boundary of $\mathcal{H}$ infinitely often.}

\begin{theorem}[\edited{Avila-Bochi-Yoccoz \cite{aby}, Theorem 4.1}]
\begin{center}
Define $\mathcal{H}\subset SL(2,\mathbb{R})^N$ to be a region of $N$-tuples of matrices such that the matrices as a set is uniform hyperbolic, and let $(A_1,\cdots,A_N)$ belong to the boundary of $\mathcal{H}$.  Then one of the following possibilities hold:
\begin{enumerate}
    \item There exists product $\prod_{ I} A_i$ which is parabolic for some ordered set $I$
    
    \item There exists product $\prod_{ I} A_i$ which is the identity for some ordered set $I$
    
    \item There exists products $\prod_{I} A_i$, $\prod_{J} A_j$, and $\prod_{K} A_k$ such that $(\prod_{I} A_i)\cdot u(\prod_{J} A_j)=s(\prod_{K} A_k)$ for some ordered sets $I$, $J$, $K$.  \edited{ Here $u(\cdot)$ is the respective unstable eigenvector and $s(\cdot)$ is the respective stable eigenvector.}
\end{enumerate}\end{center} \end{theorem}
\edited{Further research in this field can be seen in \cite{ca}.  The paper by Avila, Bocci, and Yoccoz \cite{aby} outlines an example in section 4.7 such that a path interests the boundary of $\mathcal{H}$ infinitely often. It would seem possible that for a large enough period, the spectrum may consist of an infinite number of intervals.  }

 \section{\label{sec:3}Key Lemma}
 

Theorem 1 states that if \edited{a family of ergodic Schr\"{o}dinger operators has a potential given by \hyperref[eq:2]{$(2)$}}, then the \edited{almost sure} spectrum consists of a finite number of intervals.  \edited{Applying lemma 1 with $m=2$ , we simply need to find a way to calculate the set $\{E:\{A_{E,(0,x)}^2\}\text{ is uniformly hyperbolic}\}$.  We can calculate the set using the following lemma.}  
\edited{ \label{lem:2}\begin{lemma} The set of matrices  \[\bigg\{ AC,AD,BC,BD:\begin{array}{c}A,B,C,D\text{ are transfer matrices and} \\  A_{11}\geq B_{11}\text{ and }C_{11}\geq D_{11}\end{array} \bigg\}\] is uniformly hyperbolic if and only if the matrices $AC,AD,BC,\;\&\;BD$ are individually hyperbolic.
 \end{lemma}}
 \edited{If the set is uniformly hyperbolic, all the matrices are hyperbolic.  In the regions where all the matrices are all hyperbolic, they turn out to form a principal cone.  We will identify the principal cones in the projective space \edited{in section \hyperref[sec:4]{$4$}}, proving lemma \hyperref[lem:2]{$2$}.}

 \section{\label{sec:4}Calculations}
 \edited{In t}his section, we will  \edited{provide a few technical lemmas necessary to prove} lemma \hyperref[lem:2]{$2$}.  We will begin with defining the appropriate transfer matrices and their products. Then we will define the intervals over which these products are hyperbolic.  \edited{After that, over these intervals the corresponding set of matrices will be shown to have a principal cone, proving the set is uniformly hyperbolic.}  
 \subsection{\label{sec:4A}Transfer Matrices}
 \edited{With} potential $\edited{V}$ given by \edited{the equation} \[\nu(n)=\begin{cases} \lambda_0 B(p_0)+c_0 & n\text{ is even} \\ \lambda_1 B(p_1)+c_1 & n\text{ is odd}\end{cases}.\] \edited{It is enough to consider the case where, $\lambda_0,\lambda_1,\;c_1\edited{\geq}0,$ and $c_0=0$.  For any $n$, $\nu(n)$ and $\nu(n+1)$ can take on at most 2 values, so we can define the constants} \[\edited{\begin{array}{c} c_0=\min\limits_n\{\nu(2n)\}\\ \\ \lambda_0=\max\limits_n\{\nu(2n)\}-\min\limits_n\{\nu(2n)\}\geq 0\\ \\ c_1=\min\limits_n\{\nu(2n+1)\} \\ \\ \lambda_1=\max\limits_n\{2n+1\}-\min\limits_n\{2n+1\}\geq 0\end{array}}.\]  \edited{We can also assume that $c_0\leq c_1$ without loss of generality.  If a constant is added to the potential, the spectrum of the operator is shifted by that value, so we can shift the potential by $-c_0$ without altering the spectral type or topolgical structure.  This redefines the constants $c_0=0$ and $c_1$ as a nonnegative number.} The four possible transfer matrices can now be written as 
\[\begin{array}{c}A=\begin{bmatrix} E & -1 \\ 1 & 0 \end{bmatrix}\hspace{1em} B=\begin{bmatrix} E-\lambda_0 & -1 \\ 1 & 0 \end{bmatrix}\\ \\ C=  \begin{bmatrix} E-c_1 & -1 \\ 1 & 0 \end{bmatrix}\hspace{1em} D=\begin{bmatrix} E-\lambda_1-c_1 & -1 \\ 1 & 0 \end{bmatrix}\end{array}\]
and \edited{these will be the} matrices $A,B,C,$ and $D$ in lemma \hyperref[lem:2]{$2$}.  \edited{For the context of this paper it is useful to use this expression}.  Notice that $A^2_{E,(0,x)}$ to be one of the 4 matrices $AC,AD,BC,$ or $BD$, which is the set 
\[
\scriptsize\edited{\begin{array}{c}\{AC,AD,BC,BD\}=\Bigg\{\begin{bmatrix} E & -1 \\ 1 & 0 \end{bmatrix}\cdot\begin{bmatrix} E-c_1 & -1 \\ 1 & 0 \end{bmatrix},   \begin{bmatrix} E & -1 \\ 1 & 0 \end{bmatrix}\cdot \begin{bmatrix} E-\lambda_1-c_1 & -1 \\ 1 & 0 \end{bmatrix},\\ \\     \begin{bmatrix} E-\lambda_0 & -1 \\ 1 & 0 \end{bmatrix}\cdot\begin{bmatrix} E-c_1 & -1 \\ 1 & 0 \end{bmatrix},    \begin{bmatrix} E-\lambda_0 & -1 \\ 1 & 0 \end{bmatrix}\cdot \begin{bmatrix} E-\lambda_1-c_1 & -1 \\ 1 & 0 \end{bmatrix}\Bigg\}\\ \\ =\Bigg\{\begin{bmatrix} E(E-c_1)-1 & -E \\ (E-c_1) & -1 \end{bmatrix},\begin{bmatrix} E(E-\lambda_1-c_1)-1 & -E \\ (E-\lambda_1-c_1) & -1 \end{bmatrix} \\ \\ \begin{bmatrix} (E-\lambda_0)(E-c_1)-1 & -(E-\lambda_0) \\ (E-c_1) & -1 \end{bmatrix},\begin{bmatrix} (E-\lambda_0)(E-\lambda_1-c_1)-1 & -(E-\lambda_0) \\ (E-\lambda_1-c_1) & -1 \end{bmatrix}\Bigg\}.\end{array}}\]  
In the order listed, these matrices will be denoted as $A_1,A_2,A_3,\;\&\;A_4$ \edited{such that}
\begin{equation}\label{eq:7}
\begin{array}{cc}A_1=AC &  A_2=AD\\ \\A_3=BC & A_4=BD.\end{array}\end{equation}
\edited{To prove lemma \hyperref[lem:2]{$2$},} we need to show that if each of these four matrices is hyperbolic, then the set of matrices $\{A_1,A_2,A_3,A_4\}$ is uniformly hyperbolic.  \edited{Let us recall that an $SL(2,\mathbb{R})$ matrix is hyperbolic if and only if an absolute value of its trace is greater than $2$.} In order to do that, we will identify the intervals over $E$ where the absolute value of the traces of all four matrices are greater than 2.

For each of these four matrices the eigenvectors can be easily calculated. \edited{We will benefit from the following definition and lemma first.} 
\begin{defn}\label{def:5}
\edited{ Denoted $r\in\mathbb{R}$, the \emph{signed spectral radius} of a matrix $A\in SL(2,\mathbb{R})$ satisfies}\[\begin{array}{c}sign(r)=sign(tr(A)) \\ \\ |r|=\text{spectral radius of }A \end{array}.\]
\end{defn}

\edited{In terms of $\{A_i\}$, the respective signed spectral radius will be written as $r_i$.}

\begin{lemma}\label{lem:3}
Let $r$ be a signed spectral radius of $A\edited{\in SL(2,\mathbb{R})}$.  Then the following hold:
\begin{enumerate}
\item If $A$ is hyperbolic, $r$ is the unstable eigenvalue, and $r^{-1}$ is the stable eigenvalue.
\item $r\neq 0$ as it is a root of $\lambda^2-tr(A_i)\lambda+1=0$.
\item $r>0$ if and only if the eigenvalues are positive.
\item $r<0$ if and only if the eigenvalues are negative.
\end{enumerate}
\end{lemma}
\begin{proof} \

By definition, $r$ is a root of $\lambda^2-tr(A_i)\lambda+1=0$.  Therefore $r$ is real if and only if $A$ has strictly real roots.  Because $r$ is the eigenvalue of $A$ with a larger absolute value, and so the other eigenvalue of $A$ must be the inverse of $r$ and so \[|r^{-1}|<1<|r|\]
\edited{The trace is positive if and only if the roots of the polynomial are positive.  The roots are the eigenvectors $r$ and $r^{-1}$.}

\end{proof}

\edited{We can now, more easily, explicity express the unstable $(u_i)$ and stable $(s_i)$ eigenvectors of $A_i$.}

\begin{lemma}\label{lem:4}
For any hyperbolic  $A\edited{\in SL(2,\mathbb{R})}$, denote by $r$ its signed spectral radius.  Then $r$ is a zero of its minimal polynomial, and \edited{either} the vectors
\[\begin{bmatrix} -A_{22}+r \\ A_{21} \end{bmatrix} \hspace{1em} \&\hspace{1em} \begin{bmatrix} -A_{22}+r^{-1} \\ A_{21} \end{bmatrix}\]
\edited{are eigenvectors $u_A$ and $s_A$ or one of them is the zero vector.}  

For $A_1,\;A_2,\;A_3,\;\&\;A_4$,  assuming $A_i$ is hyperbolic, the stable and unstable eigenvectors are given by:
\[ u_i=\frac{1+r_i}{(A_i)_{21}}\hspace{1em} \&\hspace{1em} s_i=\frac{1+r_i^{-1}}{(A_i)_{21}}  \]
\edited{in the projective space parameterized by $\begin{bmatrix}x \\ y\end{bmatrix}\leftrightarrow\frac{x}{y}$.}\end{lemma}
\begin{proof} \
 
 The matrix
\[A=\begin{bmatrix} A_{11} & A_{12} \\ A_{21} & A_{22} \end{bmatrix}\] \edited{is $SL(2,\mathbb{R})$}.  This means that the two eigenvalues are $r$ and $r^{-1}$ if it is hyperbolic. \edited{Because $det(A-r I)=0$, the rows of  $A-r$ are multiples of each other.  This is also true for $A-r^{-1}I.$  For all $E$ where $A_i$ is hyperbolic, the second row of $A_i-r_i$ (or $A_i-r_i^{-1}$) is not $\begin{bmatrix} 0 &0 \end{bmatrix}$. If \[\begin{bmatrix} A_{21} & A_{22}-r_i  \end{bmatrix}\neq \begin{bmatrix} 0 & 0\end{bmatrix},\] then 
\[\begin{bmatrix} A_{11}-r & A_{12} \\ A_{21} & A_{22}-r  \end{bmatrix} \cdot \vec{u} = \begin{bmatrix} 0 \\ 0 \end{bmatrix}\] if and only if  \[\begin{bmatrix} A_{21} & A_{22}-r  \end{bmatrix} \cdot \vec{u} = \begin{bmatrix} 0  \end{bmatrix}.\] This gives
\[\vec{u}\propto \begin{bmatrix} r-A_{22} \\ A_{21}\end{bmatrix}\] assuming that $A$ is hyperbolic (here $\vec{u}\propto\vec{v}$ means there is some scalar $\lambda$ such that $\vec{u}=\lambda\vec{v}$).  The same holds true for $r^{-1}$, with \[\vec{s}\propto \begin{bmatrix}r^{-1} - A_{22} \\ A_{21}\end{bmatrix}\] proving the lemma.}\end{proof}  
Because \[\edited{A_i=\begin{bmatrix} (A_i)_{11} & E \\ (A_i)_{21} & -1\end{bmatrix}}\] \edited{We can explicitly define $u_i$ and $s_i$, which we do in table \hyperref[tab:1]{$1$} below.}  For the rest of the paper, we will use the notation \edited{$r_i$ for the spectral radius of $A_i,$ $u_i\in\mathbb{P}$ for the unstable eigenvector of $A_i$, and $s_i\in\mathbb{P}$ for the stable eigenvector of $A_i$.}  Defined explicitly in the table below are the values $r_i,$ $u_i,$ and $s_i$.  
\begin{widetext}
\begin{center}\label{tab:1}
\begin{tabular}{ |p{0.5cm}|p{10.8cm}|p{1.7cm}| }
 \hline
 \multicolumn{3}{|c|}{Eigenvectors} \\ \hline 
  & $r$ & Expression \\ \hline
 $u_1$ & $r_1=\frac{2-E(E-c_1)+sgn(2-E(E-c_1))\sqrt{(2-E(E-c_1))^2-4}}{2}$ & $\frac{1+r_1}{E-c_1}$ \\ \hline
 $s_1$ & $r_1=\frac{2-E(E-c_1)+sgn(2-E(E-c_1))\sqrt{(2-E(E-c_1))^2-4}}{2}$ & $\frac{1+r_1^{-1}}{E-c_1}$ \\ \hline
 $u_2$ & $r_2=\frac{2-E(E-c_1-\lambda_1)+sgn(2-E(E-c_1-\lambda_1))\sqrt{(2-E(E-c_1-\lambda_1))^2-4}}{2}$ & $\frac{1+r_2}{E-c_1-\lambda_1}$  \\ \hline
 $s_2$ & $r_2=\frac{2-E(E-c_1-\lambda_1)+sgn(2-E(E-c_1-\lambda_1))\sqrt{(2-E(E-c_1-\lambda_1))^2-4}}{2}$ & $\frac{1+r_2^{-1}}{E-c_1-\lambda_1}$ \\ \hline
 $u_3$ & $r_3=\frac{2-(E-\lambda_0)(E-c_1)+sgn(2-(E-\lambda_0)(E-c_1))\sqrt{(2-(E-\lambda_0)(E-c_1))^2-4}}{2}$  & $\frac{1+r_3}{E-c_1}$ \\ \hline
 $s_3$ & $r_3=\frac{2-(E-\lambda_0)(E-c_1)+sgn(2-(E-\lambda_0)(E-c_1))\sqrt{(2-(E-\lambda_0)(E-c_1))^2-4}}{2}$ &  $\frac{1+r_3^{-1}}{E-c_1}$ \\ \hline
 $u_4$ & $r_4=\frac{2-(E-\lambda_0)(E-c_1-\lambda_1)+sgn(2-(E-\lambda_0)(E-c_1-\lambda_1))\sqrt{(2-(E-\lambda_0)(E-c_1-\lambda_1))^2-4}}{2}$ & $\frac{1+r_4}{E-c_1-\lambda_1}$ \\ \hline
 $s_4$ & $r_4=\frac{2-(E-\lambda_0)(E-c_1-\lambda_1)+sgn(2-(E-\lambda_0)(E-c_1-\lambda_1))\sqrt{(2-(E-\lambda_0)(E-c_1-\lambda_1))^2-4}}{2}$ & $\frac{1+r_4^{-1}}{E-c_1-\lambda_1}$ \\ \hline
     \end{tabular} 
     
\captionof{table}{For every eigenvector, there is a corresponding $r$ value it depends on.}

\end{center}

\end{widetext}
\subsection{\label{sec:4B}Intervals Over $E$}

We will now find the intervals in $E$ where all four matrices are hyperbolic.  
\begin{prop}\label{prop:1}
Over $E$, there are at most 5 intervals over which \edited{every matrix of the set $\{A_i\}$} is hyperbolic.

\end{prop} 

\edited{The set $\{E:A_i\text{ is hyperbolic for all }i\}$ is the intersection of open intervals defined by } 
\begin{equation}\bigcap_{i=1}^4\{E:|tr(A_i)|>2\},\end{equation}  \edited{and this lemma can be proven by explicitly calculating the intervals over $E$ where the statement $|tr(A_i)|>2$ is true for all $i$}. These intervals depend on the 3 possible orderings of the elements of the set $\{\lambda_0,c_1,c_1+\lambda_1\}.$ \edited{These orderings are listed below, and within the list are sublists, defining the intervals that make up the set $\{E:A_i\text{ is hyperbolic for all } i\}.$   Here, we are assuming $c_0=0,c_1\geq0,\lambda_0\geq0,\lambda_1\geq 0$,  as mentioned in section \hyperref[sec:4A]{$4A$}.} 

\begin{enumerate}\item If we assume $ \lambda_0\leq c_1\leq c_1+\lambda_1,$ \edited{then the list of intervals making up $\{E:A_i\text{ is hyperbolic for all } i\}$ are below.}
 \end{enumerate} 

\begin{enumerate}
\centering
\item[a)]  $E<0$ and $E(E-c_1)>4$

\item[b)] $0<E<\lambda_0$ and $ (E-\lambda_0)(E-c_1)>4$
\item[c)] $\lambda_0<E<c_1$
\item[d)] $c_1<E<c_1+\lambda_1$ and  $(E-\lambda_0)(E-c_1)>4 $
\item[e)]  $E>c_1+\lambda_1$ and $(E-\lambda_0)(E-\lambda_1-c_1)>4$ 
\end{enumerate}




\begin{enumerate} 
\item[2.] If $c_1< \lambda_0<c_1+\lambda_1,$ \edited{then $\{E:A_i\text{ is hyperbolic for all } i\}$ is the union of intervals listed below.}
\begin{enumerate}\centering
\item[a)] $E<0$ and $E(E-c_1)>4$
\item[b)] $0<E<c_1$ and $(E-\lambda_0)(E-c_1)>4$
\item[c)]$c_1<E<\lambda_0$ and $E(E-c_1)>4$ and $(E-\lambda_0)(E-\lambda_1-c_1)>4$
\item[d)]  $\lambda_0<E<c_1+\lambda_1$ and $(E-\lambda_0)(E-c_1)>4$
\item[e)] $E>c_1+\lambda_1$ and $(E-\lambda_0)(E-\lambda_1-c_1)>4$
\end{enumerate} \end{enumerate}  

\begin{enumerate}
\item[3.] If $c_1\leq c_1+\lambda_1\leq \lambda_0,$ \edited{the the list of intervals over which the matrices are hyperbolic are listed below.}
\begin{enumerate}\centering
    \item[a)]  $E<0$ and $E(E-c_1)>4$
    \item[b)] $0<E<c_1$ and  $(E-\lambda_0)(E-c_1)>4$
    \item[c)] $c_1<E<c_1+\lambda_1$ and $(E-\lambda_0)(E-c_1)>4$ and $E(E-c_1)>4$
    \item[d)] If $c_1+\lambda_1<E<\lambda_0$ and $E(E-\lambda_1-c_1)>4$ 
    \item[e)] If $E>\lambda_0$ and $(E-\lambda_0)(E-\lambda_1-c_1)>4$ 
\end{enumerate}\end{enumerate}
\edited{It is worth noting that these arent the explicit intervals but just defining conditions.  Depending on the constants, some of these conditions are not met for any $E$.  In the case of the trivial example, where $c_1=\lambda_0=\lambda_1=0$ (looking at either ordering 1 or ordering 3), only conditions $a$ and $e$ are possible. The explicit intervals are defined in section \hyperref[sec:4E]{$4E$}.}




\subsection{Eigenvector Combinatorics}
\edited{Before we can prove lemma \hyperref[lem:2]{$2$}}, we need some technical lemmas which can be used to locate the eigenvectors in $\mathbb{P}^1$ in relation to each other.

\begin{lemma}\label{lem:5}
Given distinct transfer matrices $A,B,C$, if the matrices $AB$ and $AC$ are hyperbolic, then they cannot share an eigenvector.  Similarly, $BA$ and $CA$ cannot share an eigenvector if they are hyperbolic.
\end{lemma}
\begin{proof}\
\begin{itemize}
\item \edited{To prove the first statement}, given matrices $AB$ and $AC$, assume they share an eigenvector $\vec{v}$.  Then
\[AB\vec{v}=\xi_1 AC\vec{v}=\xi_2 \vec{v}\]
\edited{for constants $\xi_1$ and $\xi_2$.  Therefore \[B\vec{v}=\xi_1 C\vec{v}\Rightarrow (B-\xi_1C)\vec{v}=0\]}
\edited{$det(B-\xi C)=(1-\xi_1)^2$.  For this to be zero, $\xi_1=1.$  This makes $(B-C)\vec{v}=\begin{bmatrix}B_{11}-C_{11} & 0 \\ 0 & 0 \end{bmatrix}\cdot \vec{v}=0$ in which case $\vec{v}\propto\begin{bmatrix} 0\\1\end{bmatrix}$. For this to be an eigenvalue of $AB,$ we get the equation \[AB\cdot\begin{bmatrix} 0 \\1\end{bmatrix}=A\cdot\begin{bmatrix} -1 \\0\end{bmatrix}=\xi_2\begin{bmatrix}0 \\ 1\end{bmatrix},\] which gives the equations $A=\begin{bmatrix} 0 & -1 \\ 1 & 0\end{bmatrix}$ and $\xi_2=-1$.  This makes $AB$ not hyperbolic.}

\edited{\item To prove the second statement, assume matrices $BA$ and $CA$ which share an eigenvector $\vec{v}$. \[BA\vec{v}=\xi_1CA\vec{v}=\xi_2 \vec{v}\] Therefore, \[B(A\vec{v})=\xi_1C(A\vec{v})\Rightarrow(B-\xi_1 C)(A\vec{v})=0.\]}  This gives the equation $A\vec{v}\propto\begin{bmatrix}0 \\ 1\end{bmatrix}$, which implies the equation $\xi_2\vec{v}=B\left(A\vec{v}\right)\propto B\begin{bmatrix}0 \\ 1 \end{bmatrix}\propto\begin{bmatrix}-1 \\ 0\end{bmatrix}$.  Given the statement $A\vec{v}\propto\begin{bmatrix} 0 \\ 1\end{bmatrix},$ then the equations $A=\begin{bmatrix} 0 & -1 \\ 1 & 0\end{bmatrix},$ and $\xi_2=\xi_1=1$ are true, making the product parabolic and proves the lemma.

\end{itemize}
\end{proof}

Recognizing that the matrices $\{A_i\}$ are products of transfer matrices, we can use this lemma to deduce some properties about the unstable and stable matrices.  \begin{cor}\label{cor:1}
Assuming all the matrices in $\{A_i\}$ are hyperbolic and using \hyperref[eq:7]{$(7)$}, if $A_i$ and $A_j$ share an eigenvector (with $i< j$), then $(i,j)\in\{(1,4),(2,3)\}$.
\end{cor} 

Over the intervals where the matrices are hyperbolic, the functions $u_i(E)$ and $s_i(E)$ in the projective space are well-defined and continuous over intervals of $E$.  Below we have the unstable eigenvector (red) and stable eigenvector (blue) of matrix $A_1$ in $\mathbb{P}^1$ as a function of $E$ to help illustrate what the graph can look like.

\begin{figure}[H]\label{fig:1}\centering
\includegraphics[width=6.9cm, height=4.8cm]{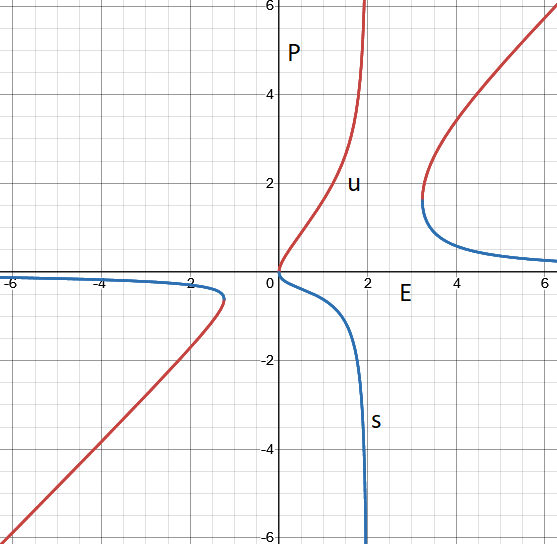}
\caption{Eigenvectors of $A_1$.  Here $c_0=0$ }
\end{figure}  

Next we have graphed the eigenvectors of matrices $A_1,\;A_2,\;A_3\;\&\;A_4$. \edited{ The matrices $A_2,A_3,\;\&\;A_4$ can be seen as perturbed versions of $A_1.$ If any of the graphs intersect they are eigenvectors of $A_2$ and $A_3$, or they are eigenvectors of $A_1$ and $A_4$.}

\begin{figure}[H]
\centering
\label{fig:2}\includegraphics[width=6.9cm, height=4.8cm]{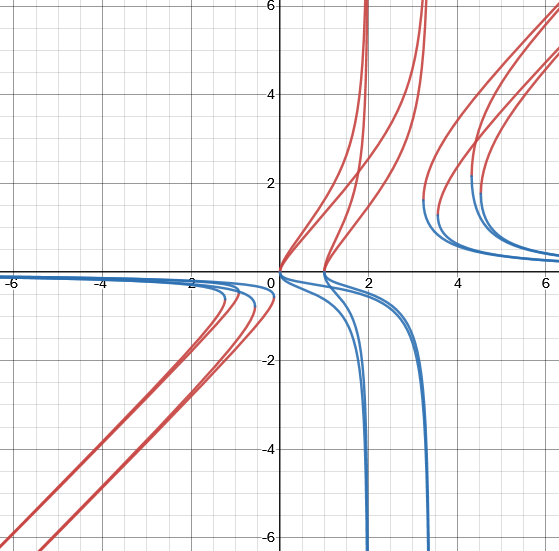} \caption{Eigenvectors of $A_1$ to $A_4$.   $\lambda_0=1,\;\lambda_1=1.4,\;c_0=0,$ and $c_1=2$}
\end{figure}

 \edited{Only a few of the eigenvector graphs can intersect, while others bound each other over intervals of $E$.  This brings us to the lemma below.}
\color{black}\begin{lemma}\label{lem:6} \
\begin{itemize} 
\item Define open interval $J$ over $E$ such that  $A_1(E)$ and $A_2(E,\lambda_1)$ are hyperbolic, and there exists $\lambda_1^*>0$ such that $A_2(E,\lambda_1)$ is hyperbolic for all $\lambda_1\in [0,\lambda_1^*]$.  Over $J$, the eigenvector $u_2$ (resp $s_2$) is greater than $u_1$ (resp $s_1$) if and only if  $\partial u_2 /\partial \lambda_1|_{\lambda_1=0}\geq0$ (resp $\partial s_2/\partial \lambda_1|_{\lambda_1=0}\geq0$).  
\item  Define open interval $J$ over $E$ such that  $A_3(E)$ and $A_4(E,\lambda_1)$ are hyperbolic, and there exists $\lambda_1^*>0$ such that $A_4(E,\lambda_1)$ is hyperbolic for all $\lambda_1\in [0,\lambda_1^*]$.  Over $J$ the eigenvector $u_4$ (resp $s_4$) is greater than $u_3$ (resp $s_3$) if and only if $\partial u_4 /\partial \lambda_1|_{\lambda_1=0}\geq0$ (resp $\partial s_4/\partial \lambda_1|_{\lambda_1=0}\geq0$). 
\item Define open interval $J$ over $E$ such that  $A_1(E)$ and $A_3(E,\lambda_0)$ are hyperbolic, and there exists $\lambda_0^*>0$ such that $A_3(E,\lambda_0)$ is hyperbolic for all $\lambda_0\in [0,\lambda_0^*]$.  Over $J$ the eigenvector $u_3$ (resp $s_3$) is greater than $u_1$ (resp $s_1$) if and only if $\partial u_3 /\partial \lambda_0|_{\lambda_0=0}\geq0$ (resp $\partial s_3 /\partial \lambda_0|_{\lambda_0=0}\geq0$).    
\item Define open interval $J$ over $E$ such that  $A_2(E)$ and $A_4(E,\lambda_0)$ are hyperbolic, and there exists $\lambda_0^*>0$ such that $A_4(E,\lambda_0)$ is hyperbolic for all $\lambda_0\in [0,\lambda_0^*]$.  Over $J$ the eigenvector $u_4$ (resp $s_4$) is greater than $u_2$ (resp $s_2$) if and only if $\partial u_4 /\partial \lambda_0|_{\lambda_0=0}\geq0$ (resp $\partial s_4 /\partial \lambda_0|_{\lambda_0=0}\geq0$).  \end{itemize}

\end{lemma}
\begin{proof} \

\begin{itemize}

    \item If $\lambda_1=0,$ then $A_1=A_2.$  Over any interval $J$ where both $A_1$ and $A_2$ are hyperbolic, both  $\partial u_2/\partial \lambda_1|_{\lambda_1=0}$ and $\partial s_2/\partial \lambda_1|_{\lambda_1=0}$ are well-defined and not constant over $J$.  Because of corollary 1, $sign(u_2-u_1)$ and $sign(s_2-s_1)$ are well-defined and fixed for all $\lambda_1\in(0,\lambda_1^*]$ over the interval $J$.
   
    \item If $\lambda_1=0,$ then $A_3=A_4.$  Over any interval $J$ where both $A_3$ and $A_4$ are hyperbolic, both  $\partial u_4/\partial \lambda_1|_{\lambda_1=0}$ and $\partial s_4/\partial \lambda_1|_{\lambda_1=0}$ are well-defined and not constant over $J$.  Because of corollary 1, $sign(u_4-u_3)$ and $sign(s_4-s_3)$ are well-defined and fixed for all $\lambda_1\in(0,\lambda_1^*)$ over $J$.
    \item If $\lambda_0=0,$ then $A_1=A_3.$  Over any interval $J$ where both $A_1$ and $A_3$ are hyperbolic, both  $\partial u_3/\partial \lambda_0|_{\lambda_0=0}$ and $\partial s_3/\partial \lambda_0|_{\lambda_0=0}$ are well-defined and not constant over $J$.  Because of corollary 1, $sign(u_3-u_1)$ and $sign(s_3-s_1)$ are well-defined and fixed for all $\lambda_0\in(0,\lambda_0^*)$ over $J$.
    \item If $\lambda_0=0,$ then $A_4=A_2.$  Over any interval where both $A_2$ and $A_4$ are hyperbolic, both  $\partial u_4/\partial \lambda_0|_{\lambda_0=0}$ and $\partial s_4/\partial \lambda_0|_{\lambda_0=0}$ are well-defined and not constant over $J$.  Because of corollary 1, $sign(u_4-u_2)$ and $sign(s_4-s_2)$ are well-defined and fixed for all $\lambda_0\in(0,\lambda_0^*)$ and over the interval.
\end{itemize}

From \edited{table \hyperref[tab:1]{$1$} in section \hyperref[sec:4A]{$4A$}} we have explicit expressions of $u_i$, and the partial derivatives are explicitly \edited{ defined in the table below.}

\begin{center}\begin{table}[H]\centering\label{tab:5}
\begin{tabular}{ |p{0.5cm}||p{3cm}|p{4cm}|}
 \hline
 \multicolumn{3}{|c|}{Partial Derivatives} \\
 \hline
  
  & $\partial/\partial\lambda_0$ & $\partial/\partial\lambda_1$ \\ \hline
 $u_1$ & 0 & 0 \\ \hline
 $s_1$ & 0 & 0 \\ \hline
 $u_2$ & 0 &  $\frac{1+(E-c_1-\lambda_1)(\partial r)/(\partial \lambda_1)}{(E-c_1-\lambda_1)^2}$  \\ \hline
 $s_2$ & 0 & $\frac{1-r^{-2}(E-c_1-\lambda_1)(\partial r)/(\partial \lambda_1)}{(E-c_1-\lambda_1)^2}$  \\ \hline
 $u_3$ & $\frac{(\partial r)/(\partial \lambda_0)}{E-c_1}$  & 0 \\ \hline
 $s_3$ &  $-\frac{(\partial r)/(\partial \lambda_0)}{r^2(E-c_1)}$    & 0 \\ \hline
 $u_4$  & $\frac{(\partial r)/(\partial \lambda_0)}{E-c_1-\lambda_1}$ &  $\frac{1+(E-c_1-\lambda_1)(\partial r)/(\partial \lambda_1)}{(E-c_1-\lambda_1)^2}$  \\ \hline
 $s_4$  &  $-\frac{(\partial r)/(\partial \lambda_0)}{r^2(E-c_1-\lambda_1)}$   &  $\frac{1-r^{-2}(E-c_1-\lambda_1)(\partial r)/(\partial \lambda_1)}{(E-c_1-\lambda_1)^2}$ \\ \hline
\end{tabular}\caption{In $\mathbb{P}$, the partial derivatives of the eigenvectors are defined here.}\end{table}
\end{center}
\end{proof}
\color{black}

\subsection{\label{sec:4D}Existence of Principal Cone }
  \edited{Here, we prove that over the intervals defined in section \hyperref[sec:4B]{$4B$,} the eigenvectors of the matrices $\{A_i\}$ form a principal cone, which will prove lemma \hyperref[lem:2]{$2$}.  The proof relies heavily} on lemmas \edited{\hyperref[lem:3]{$3$},} \hyperref[lem:4]{$4$}, \hyperref[lem:5]{$5$}, and \hyperref[lem:6]{$6$} as well as \edited{table \hyperref[tab:1]{$1$}} from the end of \edited{\hyperref[sec:4A]{$4A$}}.  \edited{We will also need the derivatives of the unstable and stable matrices with respect to $\lambda_0$ and $\lambda_1$ from \edited{ table \hyperref[tab:5]{$5$}} at the end of the proof of lemma \hyperref[lem:6]{$6$}}. Similar to subsection \edited{\hyperref[sec:4B]{$4B$}}, the \edited{3 orderings of the set $\{\lambda_0,c_1,c_1\lambda_1\}$ are all addressed separately.}
  \begin{proof} \
  
  \edited{Same as in section \hyperref[sec:4B]{$4B$}, we have a list of the different orderings of $\{c_1,\lambda_0,c+1+\lambda_1\}$, and there are sublists stating the different intervals over which the matrices are all hyperbolic.  Here, they are listed as scenarios, and an ordering of the eigenvectors are presented and then proven to be true.    
  \begin{enumerate}
      \item If we assume $\lambda_0\leq c_1$
      \begin{enumerate}
          \item[(Scenario a)]  To show $u_4<u_2,u_3<u_1<s_1<s_2,s_3<s_4<0:$ 
     \end{enumerate}
          Over this region, the trace of all matrices are greater than $2$.  Therefore, we can conclude that $r_i>1$ for all $i$, and, using \edited{lemma \hyperref[lem:3]{$3$} and table \hyperref[tab:1]{$1$}}, all the eigenvectors are negative and $u_i<s_i$.     Over this interval, we have the inequalities $\partial/\partial\lambda_i(u_j)\leq0\leq\partial/\partial\lambda_i(s_j)$ for all possible $i,j$.  \edited{From lemma \hyperref[lem:6]{$6$}, we get the}  order the eigenvectors, and hence a principal cone can be defined.
          \begin{enumerate}
    \item[(Scenario b)] To show $u_4<u_3<s_3<s_4,s_1,s_2<0<u_2<u_1:$
          \end{enumerate}
          The matrices $A_1$ \& $A_2$ have negative traces, and $A_3$ \& $A_4$ have positive traces.  This means $r_3,r_4>0>r_1,r_2$ \edited{from lemma \hyperref[lem:3]{$3$}}.  If $r_i<0,$ then $u_i$ and $s_i$ will have opposite signs \edited{which can be seen from lemma \hyperref[lem:4]{$4$} and, more explicitly, table \hyperref[tab:1]{$1$}  in subsection \hyperref[sec:4A]{$4A$}}.  Over this region we have the inequality, $u_4<u_3<s_3<s_4<0,$ because $\partial/\partial\lambda_1(u_4)\leq 0\leq\partial/\partial\lambda_1(s_4)$ \edited{and lemma \hyperref[lem:6]{$6$}.  Due to corollary \hyperref[cor:1]{$1$}}, $s_1$ is bounded by $s_3$.  It can be easily calculated that $0>s_1>s_3$, because $\lim_{E\rightarrow 0^+}s_1\nearrow 0,$ while $s_3(0)<0$ and continuous.  On top of that, $s_1<s_2$ because of lemma \hyperref[lem:6]{$6$} and the inequality $\partial/\partial\lambda_1 (s_2)\geq 0$ holds over this interval.  In the projective space, $\infty=-\infty$ is an element, and so the principal cone can be defined as $(0,\frac{u_3+s_3}{2})$, which contains the element $\infty$. 
          \begin{enumerate}
          \item[(Scenario c)] To show $s_1,s_2,s_3,s_4<0<u_1,u_2,u_3,u_4:$
          \end{enumerate}
          Over this entire interval, the traces of all four matrices are negative, making the eigenvectors different signs.  All four matrices have positive unstable eignevectors and negative stable eigenvectors.    The principal cone can be defined as $(0,\infty).$
          \begin{enumerate}
          \item[(Scenario d)]  To show $s_2,s_4<0<s_1<s_3<u_3<u_1,u_2,u_4:$
          \end{enumerate}
          The trace of $A_1$ and $A_3$ are positive, and the trace of $A_2$ and $A_4$ are negative.  This means that $r_1,r_3>1$ and $r_2,r_4<-1.$  \edited{Because of lemma \hyperref[lem:3]{$3$} and \hyperref[lem:4]{$4$}}, we get the inequalities  $s_4,s_2<0<s_1,s_3.$ Similarly, \edited{from lemma \hyperref[lem:3]{$3$}, \hyperref[lem:4]{$4$}}, and table \hyperref[tab:1]{$1$}, we get the inequalities $0<s_3<u_3$ and $0<s_1<u_1$.  To show $u_3$ serves as a lower bound for the unstable eigenvectors, the following two inequalities will suffice.  We can calculate $\partial/\partial \lambda_1(u_j)\leq 0$, so $u_3<u_1$ and $u_4<u_2$ by \edited{lemma \hyperref[lem:6]{$6$}}.  \edited{By corollary 1, $u_3$ bounds $u_4$.  As $E$ approaches the upper bound of this interval defining scenario d, $u_4\nearrow\infty$ and $u_3$ does not.  Therefore $u_3$ bounds $u_4$ from below.   This makes $u_3$} a lower bound of the unstable eigenvectors.  Because $\partial/\partial \lambda_1(s_3)>0,$ by lemma \hyperref[lem:6]{$6$}, we get the inequality $s_3>s_1$.  This is enough to give the inequalities for this scenario.
          \begin{enumerate}
          \item[(Scenario e)] To show $0<s_1<s_2,s_3<s_4<u_4<u_2,u_3<u_1:$
          \end{enumerate}
          Over this interval, assuming $(E-\lambda_0)(E-\lambda_1-c_1)>4$ which is necessary, $\partial/\partial\lambda_i(u_j)<0<\partial/\partial\lambda_i(s_j)$ for all possible $i,j$ and $s_4<u_4$.  All possible eigenvectors are positive, and\edited{ by lemma \hyperref[lem:6]{$6$}} the partial derivatives point to $s_i<s_4$ for all $i\neq 4$.  Similarly, $u_i>u_4>s_4$ for all $i\neq 4.$
  \item If we assume $c_1\leq \lambda_0<c_1+\lambda_1$ \begin{enumerate}\item[(Scenario a)] To show $u_4<u_2,u_3<u_1<s_1<s_2,s_3<s_4<0:$\end{enumerate} Here the proof is the same as scenario a  \edited{ if $\lambda_0\leq c_1$.}  \begin{enumerate}\item[(Scenario b)] To show $u_4<u_3<s_3<s_4,s_2,s_1<0<u_1<u_2:$ \end{enumerate} Here the proof is the same as scenario b  \edited{ if $\lambda_0\leq c_1$ .}
  \begin{enumerate}
  \item[(Scenario c)] Show $u_3<u_4<s_4<s_2<0<s_3<s_1<u_1<u_2: $ \end{enumerate} Over this interval, $tr(A_2),\;tr(A_3)<-2$ and $tr(A_4),\;tr(A_1)>2$, so we can determine the signs of the eigenvectors.  We get the inequalities $s_2,s_4,u_3,u_4<0<s_3,s_1,u_1,u_2$.  The matrices  $A_3$ \edited{ and $A_4$ are} hyperbolic for \begin{small}\[\edited{E\in\bigg(c_1,\frac{\lambda_0-\lambda_1-c_1-\sqrt{(\lambda_0-\lambda_1-c_1)^2-4\lambda_0(\lambda_1+c_1)+16}}{2}\bigg)}\]\end{small} and $\lim_{E\rightarrow c_1^+}u_3\searrow-\infty.$ \edited{The eigenvector $u_4$ satisfies the inequality $-\infty<u_4|_{E=c_1}<0$, however}, so we can deduce the inequalities $u_3<u_4<s_4<0$ using lemma \hyperref[lem:4]{$4$}, corollary \hyperref[cor:1]{$1$}, and $r_4>0$.  \edited{Similarly, $A_3$ and $A_1$ are hyperbolic over} \[\edited{E\in\Bigg(  \frac{c_1+\sqrt{c_1^2+16}}{2},\lambda_0   \Bigg)}\]as $\lim_{E\rightarrow \lambda_0^+}s_3\searrow 0$, \edited{but $s_1|_{E=\lambda_0}>0$ and $s_1$ is continuous.  Using lemma} \hyperref[lem:4]{$4$} and $r_1>0$, this is sufficient to show that $0<s_3<s_1<u_1$\edited{.  Additionally $A_2$ and $A_4$ are hyperbolic over} 
  \begin{small}\[\edited{E\in \bigg(0,\frac{\lambda_0-\lambda_1-c_1-\sqrt{(\lambda_0-\lambda_1-c_1)^2-4\lambda_0(\lambda_1+c_1)+16}}{2}\bigg)}.\]\end{small}\edited{Using the limit $\lim_{E\rightarrow 0}s_2\nearrow 0$, the eigenvector satisfying $s_4|_{E=0}<0,$ and continuity of $s_4$, corollary \hyperref[cor:1]{$1$} gives us $s_4<s_2<0.$  Lastly, $A_1$ and $A_2$ are hyperbolic over   \[ E\in\Bigg( \frac{c_1+\sqrt{c_1^2+16}}{2},c_1+\lambda_1\Bigg)\] }  \edited{Here, $u_2>0$ by lemma 3, $\lim_{E \rightarrow  c_1+\lambda_1^+} u_2\nearrow\infty$, and $u_1|_{E=c_1+\lambda_1}<\infty$, so by corollary 1 we have the inequality $0<u_1<u_2$}.  This completes the inequalities needed to define the principal cone. 
  \begin{enumerate} \item[(Scenario d)] To show $s_2,s_4<0<s_4<s_1,s_3<u_3<u_1,u_2,u_4:$ \end{enumerate} Here the proof is the same as scenario d  \edited{ if $\lambda_0\leq c_1$ .} \begin{enumerate} \item[(Scenario e)] To show $s_1<s_2,s_3<s_4<u_4<u_2,u_3<u_1:$ \end{enumerate} Here the proof is the same as scanrio e if $\lambda_0\leq c_1.$ \item If we assume $c_1+\lambda_1<\lambda_0$ \begin{enumerate}\item[(Scenario a)] To show $u_4<u_2,u_3<u_1<s_1<s_2,s_3<s_4<0:$ \end{enumerate} Here the proof is the same as scenario a \edited{ if $\lambda_0\leq c_1.$}  \begin{enumerate}\item[(Scenario b)] To show $u_4<u_3<s_3<s_4,s_2,s_1<0<u_1<u_2:$ \end{enumerate} Here the proof is the same as scenario b  \edited{ if $\lambda_0\leq c_1$ .}  
  \begin{enumerate}\item[(Scenario c)] Show $u_3<u_4<s_4<s_2<0<s_3<s_1<u_1<u_2:$ \end{enumerate} Here the proof is the same \edited{ if $c_1\leq\lambda_0< c_1+\lambda_1$} in scenario c.  
  \begin{enumerate}
 \item[(Scenario d)] To show $u_4<u_3<0<s_1,s_3,s_4<s_2<u_2<u_1:$\end{enumerate} The spectral radii satisfy the inequality $r_3,r_4<0<r_1,r_2$, So we can conclude that $u_4,u_3<0$, but all the other eigenvectors are greater than $0.$   \edited{By lemma 4}, we can also conclude that $s_2<u_2$.  By lemma \hyperref[lem:6]{$6$} and $0<\partial/\partial\lambda_1(s_2)$, we can deduce that $s_1<s_2$ \edited{by lemma 4}.  \edited{Given that }$\partial/\partial\lambda_1(s_4)>0$, so $s_3<s_4$ \edited{ by lemma \hyperref[lem:6]{$6$}}.  Given the limit $\lim_{E\rightarrow \lambda_1+c_1^+}s_4\searrow 0$, as well as the inequality $s_2|_{E=\lambda_1+c_1}>0,$ with $s_2$ being continuous, we have the inequality $0<s_4<s_2$ by corollary \hyperref[cor:1]{$1$}.  Last, we need the inequality $u_2<u_1$, which is true by $\partial/\partial\lambda_1(u_2)<0$ \edited{ by lemma \hyperref[lem:6]{$6$}}.   \vspace{1em}\begin{enumerate}\item[(Scenario e)] To show $s_1<s_2,s_3<s_4<u_4<u_2,u_3<u_1:$ \end{enumerate} Here the proof is the same as scenario e \edited{ if $\lambda_0<c_1.$}  \end{enumerate}}
  
  \end{proof}

\begin{widetext}
\subsection{\label{sec:4E}The Explicit Gaps in the Spectrum}

We can conclude that for all of the intervals listed as scenarios $a$ to $e$, the 4 matrices $A_1,A_2,A_3,\&A_4$ are hyperbolic and the set is uniformly hyperbolic.  The spectrum therefore consists of at most 4 intervals, which are the complement of the intervals in $E$ over which the set of matrices is uniformly hyperbolic.  Below the intervals are explicitly calculated.

\begin{prop}
\edited{The complement of the spectrum of $H_V$ from theorem 1 can be explicitly calculated, depending on the ordering of the set $\{\lambda_0,c_1,c_1+\lambda_1\},$ where $\lambda_i,c_1\geq 0$}

\begin{enumerate}
\item \edited{If we assume $ \lambda_0 \leq c_1 \leq c_1+\lambda_1$, then }

\[\edited{\mathbb{R} \backslash \sigma(H_V)}= \bigg(-\infty,\frac{1}{2}\Big(c_1-\sqrt{c_1^2+16}\big)\Bigg)\cup \bigg(0,\frac{1}{2}\Big(\lambda_0+c_1-\sqrt{(\lambda_0+c_1)^2-4(c_1\cdot\lambda_0-4)}\Big)\bigg)\cup (\lambda_0,c_1)\cup \]\[\bigg( \frac{1}{2}\Big(\lambda_0+c_1+\sqrt{(\lambda_0+c_1)^2-4(c_1\lambda_0-4)}\Big),c_1+\lambda_1\bigg)\cup\Big(\frac{1}{2}\Big(\lambda_0+\lambda_1+c_1+\sqrt{(\lambda_0+\lambda_1+c_1)^2-4\lambda_0\cdot(\lambda_1+c_1)+16},\infty\bigg)\]

\item If we assume $ c_1<\lambda_0 <  c_1+\lambda_1$, then then gaps in the spectrum consist of the intervals
\[ \mathbb{R} \backslash \sigma(H_V)=\bigg(-\infty , \frac{1}{2}\Big(c_1-\sqrt{c_1^2+16} \Big)\bigg)\cup\bigg(0 , \frac{1}{2}\Big(\lambda_0+c_1-\sqrt{(\lambda_0+c_1)^2-4(c_1\lambda_0-4)}\Big) \bigg)\cup\]\[\bigg(\frac{1}{2}\Big(c_1+\sqrt{c_1^2+16}\Big) ,  \frac{1}{2}\Big(\lambda_0+\lambda_1+c_1-\sqrt{(\lambda_0+\lambda_1+c_1)^2-4\lambda_0(\lambda_1+c_1)+16}\Big)\bigg)\cup\]\[\bigg( \frac{1}{2}\Big(\lambda_0+c_1+\sqrt{(\lambda_0+c_1)^2-4(c_1\lambda_0-4)}\Big) , c_1+\lambda_1  \bigg)\cup\bigg(\frac{1}{2}\Big(\lambda_0+\lambda_1+c_1+\sqrt{(\lambda_0+\lambda_1+c_1)^2-4(\lambda_0(\lambda_1+c_1)-4)}),\infty \bigg)\]  

\item If we assume $  c_1 \leq c_1+\lambda_1 \leq \lambda_0$, then the gaps in the spectrum consist of the intervals
\[\mathbb{R} \backslash \sigma(H_V)=\bigg(-\infty , \frac{1}{2}\Big(c_1-\sqrt{c_1^2+16}\Big) \bigg)\cup\bigg(0 , \frac{1}{2}\Big(\lambda_0+c_1-\sqrt{(\lambda_0+c_1)^2-4(c_1\lambda_0-4)}\Big) \bigg)\cup\]
\[\bigg(\frac{1}{2}\Big(c_1+\sqrt{c_1^2+16}\Big) ,  \frac{1}{2}\Big(\lambda_0+\lambda_1+c_1-\sqrt{(\lambda_0+\lambda_1+c_1)^2-4(\lambda_0(\lambda_1+c_1)-4)}\Big)\bigg)\cup\]\[\bigg(\frac{1}{2}\Big(c_1+\lambda_1+\sqrt{(c_1+\lambda_1)^2+16}\Big),\lambda_0\Bigg)\cup\bigg(\frac{1}{2}\Big(\lambda_0+\lambda_1+c_1+\sqrt{(\lambda_0+\lambda_1+c_1)^2-4(\lambda_0(\lambda_1+c_1)-4)}\Big),\infty\bigg)\]
\end{enumerate}
\end{prop}
     \edited{Note that in all scenarios, some of the intervals may be the empty set if the lower bound is greater than or equal to the upper bound.}  
\edited{We can conclude that $H_\nu-E$ has a bounded inverse if and only if $E$ is not in one of these, and the spectrum consists of the complement of these 5 intervals.  }

\edited{This proposition follows from lemma 1 by direct computation.}

\section{Acknowledgements}
I would like to acknowledge Anton Gorodetski for helping edit and put together this paper.  \edited{The author was supported in part by NSF grant DMS--1855541 (PI A.Gorodetski).}

\end{widetext}


\begin{thebibliography}{9}
  
  
  
  
    \bibitem{aby} A. Avila, J. Bochi, and J.-C. Yoccoz, Uniformly hyperbolic finite-valued $SL(2,\mathbb{R})$-cocycles, {\it Comment. Math. Helv.}, vol. 85 (2010), no. 4, pp. 813--884.
  
  
    
        \bibitem{AJ} A.\,Avila, S.\,Jitomirskaya, The Ten Martini Problem, {\it Ann. of Math.}, vol. 170 (2009), no. 1, pp. 303--342.
  
    
    \bibitem{ca} A.\,Christodoulou, Parameter Spaces of Locally Constant Cocycles, {\it International Mathematics Research Notices}, 2021;, rnab116, https://doi.org/10.1093/imrn/rnab116




     \bibitem{d1} D.\,Damanik,  Schr\"odinger operators with dynamically defined
potentials,  {\it Ergod. Th. \& Dynam. Sys.}, vol. 37 (2017) , pp. 1681-1764.
    
 
    \bibitem{d4} D.\,Damanik, M.\,Embree, A.\,Gorodetski, Spectral Properties of Schr\"odinger Operators Arising in the Study of Quasicrystals, {\it Mathematics of Aperiodic Order}, (2015) pp. 307--370, Birkhäuser, Basel.



    
     \bibitem{d2} D.\,Damanik, J.\,Fillman, Spectral Properties of Limit-Periodic Operators, 2018; arXiv:1802.05794. 
     
     
     \bibitem{d5} D.\,Damanik, J.\,Fillman, P.\,Gohlke, Spectral Characteristics of Schr\"{o}dinger Operators Generated by Product Systems, preprint.
    
    \bibitem{d3}  D.\,Damanik, Z.\,Gan,  Spectral Properties of Limit-Periodic Schr\"odinger Operators, {\it Commun. Pure Appl. Anal.}  vol 10. (2009), no. 3, pp. 859-871; arXiv:0906.3337.
    

 \bibitem{a} A. Gorodetski, V. Klepstyn, Non-stationary versions of Anderson Localization and Furstenberg Theorem on random matrix products, \edited{unpublished}.
 
 
 
 
\bibitem{mj} C. Marx, S. Jitomirskaya, Dynamics and spectral theory of quasi-periodic Schrödinger-type operators. {\it Ergodic Theory and Dynamical Systems}, vol.37 (2017), no. 8, pp. 2353-2393. doi:10.1017/etds.2016.16

  
   
   
       \bibitem{s} G. Stolz, An introduction to the mathematics of Anderson localization, {\it Entropy and the quantum II. Contemp. Math} 551, 201 pp. 71-1081.
 
 
 \bibitem{jo}\edited{ R. Johnson, Exponential dichotomy, rotation number, and linear differential operators with bounded coef-
ficients, {\it J. Differential Equations} vol. 61 (1986), pp. 54–78.}
 
     
    \bibitem{Suto} A.\,Sut\"o, The spectrum of a quasiperiodic Schr\"odinger operator, {\it Comm. Math. Phys.}, vol. 111 (1987), no. 3, pp. 409–-415.
    
 
 
    
    \bibitem{y} J.-C. Yoccoz, Some questions and remarks about $SL(2,\mathbb{R})$ cocycles, {\it Modern Dynamical Systems and Applications}, (2004) pp. 447–458, Cambridge Univ. Press, Cambridge 
   

        \bibitem{z}  Z.\,Zhang, Resolvent set of Schr\"{o}dinger operators and uniform hyperbolicity,  (2013), \edited{unpublished}, arXiv:1305.4226.

    \bibitem{z2} \edited{ Z. Zhang, Uniform hyperbolicity and its relation with spectral analysis of 1D discrete Schrödinger operators. {\it J. Spectr. Theory} vol 10 (2020), no. 4, pp. 1471–1517
}


\end{thebibliography}
\end{document}